\documentclass[a4paper,11pt]{article}
\usepackage{tikz}

\usetikzlibrary{matrix}
\usepackage[all]{xy}
\usetikzlibrary{patterns}
\usepackage[utf8x]{inputenc}
\usepackage[english]{babel}
\usepackage{amsmath,mathtools,bm}
\usepackage{amssymb}
\usepackage{amsthm}
\usepackage{verbatim}
\usepackage{graphicx}
\usepackage[small]{caption}
\usepackage[all]{xy}

\usepackage[numbers]{natbib}
\usepackage{enumitem}

\usepackage{tikz-cd }             
\usetikzlibrary{patterns}
\usepackage{graphicx}

\theoremstyle{definition}
\newtheorem{definizione}{Definition}[section]

\theoremstyle{plain}
\newtheorem{teorema}[definizione]{Theorem}
\newtheorem*{teorema*}{Theorem}
\theoremstyle{plain}
\newtheorem{lemma}[definizione]{Lemma}
\newtheorem*{property*}{Property 1}
\theoremstyle{plain}
\newtheorem{corollario}[definizione]{Corollary}
\theoremstyle{plain}

\theoremstyle{definition}

\theoremstyle{plain}

\theoremstyle{plain}

\newcommand{\ZZ}{{\mathbb Z}}

\usepackage{subcaption}

\newenvironment{salign*}{\begin{equation*}\begin{aligned}}{
	\end{aligned}\end{equation*}\ignorespacesafterend}
\usepackage{csgbedit}
\usepackage{blindtext}
\title{Cohomology of generalised configuration spaces of points on $\mathbb{R}^r$}	\author{Marcel B\"{o}kstedt, Erica Minuz\footnote{The authors are based in Aarhus University, Science and Thechnology, Department of Mathematics. The work is supported by \textit{Det frie forskningsr\aa{}d, natur og univers} $6108-00539A$.}}
\date\today

\newcommand\Conf{\mathrm{Conf}}

\begin{document}
	\maketitle
	\textbf{Abstract.} We compute the cohomology ring of a generalised type of configuration space of points in $\mathbb{R}^r$. This  configuration space is indexed by a graph. In the case the graph is complete the result is known and it is due to Arnold and Cohen. However, our computations give a generalisation to any graph and an alternative proof of the classical result. Moreover, we show that there are deletion-contraction short exact sequences for this cohomology rings. 
\tableofcontents

	\section{Introduction}
The
	configuration space of $n$ points in $\mathbb{R}^r$, that we denote as
	$\Conf_r(n)$, is defined as 	\begin{equation*}
	\Conf_r(n)=\{(x_1,\dots,x_n)\in\mathbb{R}^{rn};x_i\neq x_j\text{ for }i\neq j, 0<i,j\leq n\}.
	\end{equation*} Its cohomology ring has been computed 
        by  Arnold
	\cite{Arnol'd1969} in the case $r=2$ and Cohen
	\cite{COHEN199519} for $r\geq 3$ and it is given by the following quotient of graded rings
	\begin{equation*}
	H^\ast(\Conf_r(n))=\mathbb{Z}[e_{\alpha_{i,j}}]/{\sim}
	\end{equation*}
	where $0<i,j\leq n$, $\mathbb{Z}[e_{\alpha_{i}}]$ is the free
	commutative graded algebra generated by $e_{\alpha_{i,j}}$ of
	degree $r-1$ and $\sim$ are the relations
	\begin{itemize}
		\item $e_{\alpha_{i,j}}=(-1)^{r}e_{\alpha_{j,i}}$
		\item $e_{\alpha_{i,j}}^{2}=0$ if $r$ is odd
		\item
		$e_{\alpha_{a,b}}e_{\alpha_{b,c}}+e_{\alpha_{b,c}}e_{\alpha_{c,a}}+e_{\alpha_{c,a}}e_{\alpha_{a,b}}=0$
	\end{itemize}

In this paper we will study a  generalisation of the definition of
the configuration spaces $\Conf_r(n)$ to configuration spaces depending on a graph. These were 
	defined by \emph{Eastwood} and \emph{Huggett} \cite{E-H}, and
	also described in \cite{B-S}. Given a graph $\Gamma$, we denote the configuration space of points in $\mathbb{R}^r$ depending on a graph by $\Conf_r(\Gamma)$. Let $\alpha_{i,j}$ denote an edge in $\Gamma$ between the vertices $v_i$ and $v_j$, the generalised configuration space of points in $\mathbb{R}^r$ is is defined as 
	\begin{equation*}
	\Conf_r(\Gamma)=\{(x_1,\dots,x_n)\in\mathbb{R}^{rn};x_i\neq x_j\text{ if }\alpha_{i,j}\text{ is an edge in }\Gamma\}.
	\end{equation*}

The main result in the article is provided by the computation of the cohomology of $\Conf_r(\Gamma)$ for any graph $\Gamma$. The cohomology of $\Conf_r(\Gamma)$ is given by a commutative graded ring that depends on the parity of the integer $r$. Let 
	$\mathbb{Z}[e_{\alpha_{i}}]$ be  the free commutative graded
	algebra generated by $e_{\alpha_{i,j}}$ of degree $r-1$,
	where $\alpha_{i,j}$ is an edge in $\Gamma$ between the
	vertices $i$ and $j$ oriented from $i$ to $j$. Let $w$ be a
	circuit in $\Gamma$, that is  graph consisting  of an ordered sets  of edges $w_1,w_2,\dots ,w_k$ and vertices $v_1(w),\dots ,v_k(w),v_{k+1}(w)=v_1(w)$ such that $v_i(w),v_{i+1}(w)$ are the two  vertices incident to $w_i$. We denote by
	$e_{w}=e_{v_{1,2}}\cdot e_{v_{2,3}}\ldots\cdot e_{v_{s,1}}$ the product of the
	generators corresponding to the edges in the cycle $w$. We prove that the cohomology ring is given by the following quotient of graded rings
	\begin{equation*}
H^\ast(\Conf_r(\Gamma))=\mathbb{Z}[e_{\alpha_{i,j}}]/{\sim}
	\end{equation*}
	where $\sim$ are the relations
	\begin{itemize}
		\item $e_{\alpha_{i,j}}=(-1)^{r}e_{\alpha_{j,i}}$
		\item $e_{\alpha_{i,j}}^{2}=0$ if $r$ is odd
		\item
		$A(w)=\sum_{i}(-1)^{(r-1)i}e_{v_{1,2}}\cdots\hat{e}_{v_{i,j}}\cdots
		e_{v_{s_j,1}}=0$ for every circuit $w$ in $\Gamma$.
	\end{itemize}
	
	\noindent We call the relations $A(w)$ \emph{generalised Arnold relations}. Moreover, let $\Gamma\smallsetminus a$ denote the graph obtained from $\Gamma$ by deleting the edge $\alpha$,  and $\Gamma/\alpha$ the graph obtained by contracting the edge $\alpha$. There is a deletion-contraction short exact sequence in cohomology
	\begin{equation*}
	0\rightarrow
	H^{\ast}(\Conf_r(\Gamma\smallsetminus e))\rightarrow
	H^{\ast}(\Conf_r(\Gamma)) \rightarrow
	H^{\ast-r+1}(\Conf_{r}(\Gamma/e))\rightarrow 0.
	\end{equation*}

        Our approach is as follows. In section~\ref{sec:Alg} we define a graded commutative ring depending on a graph $\Gamma$. The definition is purely algebraic. There are two cases, depending on whether the generators are in even or odd degrees. We establish a deletion-contraction exact sequence for these rings. Such long exact deletion-contraction sequences are well known in graph cohomology, but our main algebraic result is that these long exact sequences actually break up into short exact sequences. 

        In section~\ref{sec:Space} we show that the cohomology rings of the graph indexed configuration space of points in an open discs are given by the rings defined in the previous section. The method is to first establish deletion-contraction long exact sequences of cohomology of generalized configuration spaces, then  to show that the algebraic exact sequences of the previous paragraph maps to these short exact sequences by surjections, and conclude that these surjections are actually isomorphisms by induction over the number of edges in the graphs. This method also gives an alternative approach to the computation of the cohomology of  not generalized configuration spaces.
        
 The results here presented are part of the second author's Ph.D. thesis \textit{Graph complexes and cohomology of configuration spaces}, supervised by the first author.

	\section{Algebraic description of $R^{r}(\Gamma)$}
        \label{sec:Alg}
        In this section we  will describe a graded commutative ring
        $R^r(\Gamma)$. If $r$ is even, $\Gamma$ can be any not oriented graph.   If $r$ is odd, we demand that each edge of $\Gamma$ come with an orientation, that is an ordering of its two adjacent vertices. In both cases we will assume that $\Gamma$ does not have loops, but we do allow multiple edges.
        
        We introduce some notation. A circuit in a graph consists of ordered sets of edges $w_1,w_2,\dots ,w_k$ and vertices $v_1(w),\dots ,v_k(w),v_{k+1}(w)=v_1(w)$ such that $v_i(w),v_{i+1}(w)$ are the two  vertices incident to $w_i$.     
        In case $r$ is odd, the circuit $w$ comes with an additional signs  $\epsilon_i(w)$.  The given orientation of the edge $w_i$ determines an order of the pair of vertices $v_i(w),v_{i+1}(w)$, If in this order $v_i<v_{i+1}$, then $\epsilon_i(w)=1$, else $\epsilon_i(w)=-1$.

      For a circuit $w$ of length $l(w)$, we denote by
$e_{w}=w_1\cdot w_2\ldots\cdot w_{l(w)}$  the product of the generators
corresponding to the edges in the circuit $w$.  If one changes the
order of $w$ by a cyclic permutation to obtain a new circuit $w'$, one
changes $e_w$  by at most by a sign: $e_w = \pm e_{w'}$.       
        
For $r$ and $\Gamma$ we will define a graded commutative algebra
$\Lambda^r(\Gamma)$, and an ideal $I^r(\Gamma)$ in this algebra. The
precise definitions depend on whether $r$ is even or odd.
    
\begin{definizione}\label{R_r}
  If $r$ is even, the algebra $\Lambda^r(\Gamma)$ is the free graded
  commutative algebra over the integers with one generator $e_\alpha$
  in degree $r-1$ for each edge $\alpha\in E(\Gamma)$.  If $r$ is odd,
  $\Lambda^r(\Gamma)$ is the quotient of the graded commutative
  algebra over the integers with one generator $e_\alpha$ for each
  edge $\alpha\in E(\Gamma)$ by the relations $e_\alpha^2=0$. In this
  case, $\Lambda^r[\Gamma]$ is actually commutative.  For each circuit
  we define its Arnold class:
  \[
    A(w)=
    \begin{cases}
      A(w) = \sum_{i}(-1)^{i}w_1\cdots\hat{w_i}\cdots
      w_{l(w)}&\text{ if $r$ is even,}\\
      A(w) = \sum_{i}\epsilon_i(w)w_1\cdots\hat{w_i}\cdots
      w_{l(w)}&\text{ if $r$ is odd.}
    \end{cases}
  \]
  In case $w$ consists of a single edge, which then has to be loop,
  this will be interpreted as $A(w)=1$.  Let the generalized Arnold
  ideal $I^r(\Gamma)$ be the ideal of $\Lambda^r(\Gamma)$ generated by
  the Arnold classes. Finally, let $R^r(\Gamma)$ be the quotient ring
  $\Lambda^r(\Gamma)/I^r(\Gamma)$.
\end{definizione}
              
  A map of graphs $f:\Gamma\to \Gamma'$ induces a map $f^R:\Lambda^r(\Gamma) \to \Lambda^r(\Gamma)$ which preserves the Arnold classes, so it also induces a map of rings $f^R:R^r(\Gamma)\to R^r(\Gamma')$.
	
  We can usually assume that $\Gamma$ has no multiple edges, since the
  following lemma holds.

  \begin{lemma}\label{le:double_edges}
    Let $\Gamma$ be a graph and $e$ an edge of $\Gamma$ such that
    there exists a different edge $e'$ incident to the same vertices
    as $e$. Let $\Gamma'=\Gamma\smallsetminus e$. Then, the rings $R^{r}(\Gamma)$
    and $R^{r}(\Gamma')$ are isomorphic.
  \end{lemma}
  \begin{proof}
    There is an inclusion of graphs $i:\Gamma'\to \Gamma$, and a left
    inverse $p$ to $i$, such that $p(e)=e'$.  In $R^{r}(\Gamma)$ we
    have the generalised Arnold relation $e'-e=0$, so that
    $e'=e\in R^r(\Gamma)$. It follows easily that the induced maps
    $i_*:R^r(\Gamma')\to R^r(\Gamma)$ and
    $p_*:R^r(\Gamma)\to R^r(\Gamma')$ are inverse isomorphisms.
  \end{proof}
	
	\subsection{Deletion-contraction short exact sequence for
		$R^{r}(\Gamma)$}
	Let $\Gamma$ be a graph and
	$\alpha\in E(\Gamma)$. We will mainly be interested in graphs without loops and multiple edges, but it is convenient not to exclude these cases, in order to be able to formulate certain induction arguments in a smooth way.
        We can delete the edge $\alpha$ from the graph $\Gamma$ to obtain the graph $i_\alpha:\Gamma\smallsetminus \alpha\subset \Gamma$. We can also contract the edge $\alpha$ to obtain a graph $\Gamma/\alpha$. The graph $\Gamma/\alpha$ might have multiple edges, but it does  not have loops. There is a map $p_\alpha:\Gamma\to \Gamma/\alpha$ which identifies the two vertices incident to $\alpha$.
There are induced maps of graded algebras $i_\alpha^\Lambda:\Lambda^r[\Gamma\smallsetminus \alpha]\to \Lambda^r(\Gamma)$ and $p_\alpha^\Lambda:\Lambda^r[\Gamma]\to \Lambda^r[\Gamma/\alpha]$.
        
        If $\beta\in E(\Gamma\smallsetminus \alpha)$ we alternatively denote its image $i_\alpha(\beta)\in E(\Gamma)$ by $\beta$. If $\gamma\in E(\Gamma)$, we alternatively denote its image in $p_\alpha(\gamma)\in E(\Gamma/\alpha)$ by $[\gamma]$.

        Let $\Lambda[e_{\alpha}]$ be the
	exterior algebra on the generator $e_{\alpha}$ of degree $r-1$.

        \begin{definizione}
          We consider the following two ring homomorphisms:
        \begin{align*}
          \iota_\alpha:\Lambda^r[\Gamma/\alpha] &\to \Lambda^r(\Gamma/\alpha)\otimes \Lambda^r[e_\alpha] &\quad \psi_\alpha:\Lambda^r(\Gamma)&\to \Lambda^r(\Gamma/\alpha)\otimes \Lambda^r[e_\alpha]\\
          \iota_\alpha(e_{[\eta]})&=e_{[\eta]}\otimes 1&
                                                          \psi_\alpha(e_{[e_\eta]})&=
        \begin{cases}
          e_{[\eta]}\otimes 1&\text{ if $\eta\neq \alpha$,}\\
           1\otimes e_{[\alpha]}&\text{ if $\eta= \alpha$.}
        \end{cases}
        \end{align*}
        \end{definizione}

        Consider the generalized Arnold classes in $\Lambda^r(\Gamma/\alpha)\otimes \Lambda^r[e_\alpha]$. 
We define the Arnold ideal in $\Lambda^r(\Gamma/\alpha)\otimes \Lambda^r[e_\alpha]$ to be the ideal generated by the generalized Arnold classes in $\Lambda^r(\Gamma/\alpha)\subset \Lambda^r(\Gamma/\alpha)\otimes \Lambda^r[e_\alpha]$. If $\Gamma'$ is the graph obtained from $\Gamma/\alpha$ by adding a single vertex and a single edge connecting the new vertex to $\alpha$, there is an obvious isomorphism preserving Arnold ideals between this ring and $\Lambda(\Gamma')$.
	\begin{lemma}
          \label{le:exterior.square}
		The following diagram of ring maps is commutative and the maps preserve the generalized Arnold ideals.
	\begin{equation*}
	\begin{tikzcd}
	\Lambda^r(\Gamma\smallsetminus\alpha)\ar[r,"p_\alpha^\Lambda\circ i_\alpha^\Lambda"]\ar[d,"i_\alpha^\Lambda"] & \Lambda^r(\Gamma/\alpha)\ar[d,"\iota_\alpha^\Lambda"]\\
	\Lambda^r(\Gamma)\ar[r,"\psi_\alpha^\Lambda"] & \Lambda^r(\Gamma/\alpha)
	\otimes \Lambda^r[e_\alpha]
	\end{tikzcd}
      \end{equation*}
              \end{lemma}
              \begin{proof}
                The diagram commutes since if $e_\eta\in E[\Gamma\smallsetminus \alpha]$ is a generator both paths to the lower right square takes it to $[\beta]\otimes 1$.  $i_\alpha^\Lambda$ and $p_\alpha^\Lambda$ takes an Arnold class to an Arnold class, since they are induced by maps of graphs. The map $\iota_\alpha^\Lambda$ obviuously preserves the Arnold elements, so we only need to check that $\psi_\alpha^\Lambda$ does. Let $w$ be a circuit in $\Gamma$. We have to prove that $\psi_\alpha^\Lambda(A(w))$ is contained in the ideal generated by $A(u)\otimes 1$ for $u$ a circuit in $\Gamma/\alpha$. There are three cases. If $w_i \neq \alpha$ for all $i$, then
                $u=p_\alpha^\Lambda(w)$ is a circuit in $\Gamma/\alpha$ such that $p_\alpha^\Lambda(A(w))=A(u)$, and we are done. If $\alpha$ occurs more than once in $w$, then $A(w)=0$, and we are done again. If finally $e_i=\alpha$ for a unique $i$, then
$u=[w_1][w_2]\dots \widehat{[w_i]}\dots [w_{l(w)}]$ is a circuit in $\Gamma/\alpha$ such that $\psi_\alpha^\Lambda(A(w))=A(u)$, and all  our work is done.               
	\end{proof}
        It follows from lemma~\ref{le:exterior.square} that we get an induced commutative diagram of ring maps
		
		\begin{equation*}
		\begin{tikzcd}
		R^r(\Gamma\smallsetminus\alpha)\ar[r,"p_\alpha^R\circ i_\alpha^R"]\ar[d,"i_\alpha^R"] & R^r(\Gamma/\alpha)\ar[d,"\iota_\alpha^R"]\\
		R^r(\Gamma)\ar[r,"\psi_\alpha^\Lambda"] & R^r(\Gamma/\alpha) \otimes
		\Lambda^r[e_\alpha]
		\end{tikzcd}
		\end{equation*}

	\begin{teorema}\label{theorem1}
		The above diagram
		is a pullback diagram. The map
		$
		i_\alpha^R:R^r(\Gamma\smallsetminus\alpha)\rightarrow R^r(\Gamma)
	        $
		is injective, and the map
                $\psi_\alpha^R:R^r(\Gamma)\to R^r(\Gamma/\alpha)\otimes \Lambda^r[e_\alpha]$ is surjective.
              \end{teorema}
              We will prove this theorem in the nex subsection.
	As consequences we have
	\begin{corollario}
		Suppose that $\Gamma'$ is a subgraph of $\Gamma$ such that
		$V(\Gamma')=V(\Gamma)$. The map induced by inclusion 
		$ i_*:R^r(\Gamma')\rightarrow R^r(\Gamma)$ is injective.
	\end{corollario}
	\begin{corollario}\label{corollario 1}
		For every $\alpha\in E(\Gamma)$ there is a short exact
		sequence of Abelian groups
		\begin{equation*}
		0\rightarrow R^r(\Gamma\smallsetminus\alpha)_{k}\rightarrow R^r(\Gamma)_{k}\rightarrow R^r(\Gamma/\alpha)_{k-r+1}\rightarrow 0
		\end{equation*}
		where the indices $k$  and $k-r+1$ denote the grading in the ring.
	\end{corollario}
	
	\subsection{Proof of \emph{Theorem }\ref{theorem1}}\label{proof}
        Let $\Lambda^r[\Gamma/\alpha]\otimes e_\alpha$ denote the ideal in
        $\Lambda^r[\Gamma/\alpha]\otimes \Lambda^r[e_\alpha]$ generated by $e_\alpha$.
        As an Abelian group, $\Lambda^r[\Gamma/\alpha]\otimes \Lambda^r[e_\alpha]$ is the direct sum of the image of the injective map $\iota_\alpha^\Lambda$ and
the ideal        $\Lambda^r[\Gamma/\alpha]\otimes e_\alpha$. Let
        $\pi:\Lambda^r[\Gamma/\alpha]\otimes \Lambda[e_\alpha]
        \to \Lambda^r[\Gamma/\alpha]\otimes e_\alpha$ be the projection, and define
        \[
          g_\alpha^\Lambda=\pi\circ \psi_\alpha^\Lambda:\Lambda[\Gamma]\to
          \Lambda^r[\Gamma/\alpha]\otimes e_\alpha.
        \]        
Since $\psi_\alpha^\Lambda$ and $\pi$ preserve the ideal generated by the Arnold classes, so does $g_\alpha^\Lambda$. We obtain a restricted map
$g_\alpha^I :I(\Gamma)\to I(\Gamma/\alpha)\otimes e_\alpha$ and a quotient map
$g_\alpha^R:R^r(\Gamma) \to R^r(\Gamma/\alpha)\otimes e_\alpha$. 
We also immediately obtain a commutative diagram:

\begin{equation}\label{diagram}
	\begin{tikzcd}[column sep=1.5em]
	&0\ar[d]&0\ar[d]&0\ar[d]&\\
	0\ar[r] &I^r(\Gamma\smallsetminus\alpha) \ar[r]\ar[d,"i_\alpha^I"]
	&\Lambda^r[\Gamma\smallsetminus\alpha] \ar[r]\ar[d,"i_\alpha^\Lambda"]
	&R^r(\Gamma\smallsetminus\alpha) \ar[r]\ar[d,"i_\alpha^R"]
 	&0\\
	0\ar[r] &I^r(\Gamma)\ar[r]\ar[d,"g_\alpha^I"] &
	\Lambda^r[\Gamma]\ar[r]\ar[d,"g_\alpha^\Lambda"] &
	R^r(\Gamma)\ar[r]\ar[d,"g_\alpha^R"]
	&0\\
	0\ar[r]& I^r(\Gamma/\alpha)\otimes e_{\alpha}\ar[r]\ar[d] &
	\Lambda^r[\Gamma/\alpha]\otimes e_{\alpha} \ar[r]\ar[d] &
	R^r(\Gamma/\alpha)\otimes e_{\alpha} \ar[r]\ar[d]
	&0\\
	&0&0&0&
	\end{tikzcd}
      \end{equation}
      The rows of this diagram are short exact by definition. Most of this subsection
      will go into proving that the three columns are short exact.
	\begin{lemma}
          \label{le:middle}
	The middle column in diagram~\ref{diagram} is  exact.
	\end{lemma}
	\begin{proof}
          The inclusion $i_\alpha^\Lambda $ is clearly injective by its definition.
          
          An element in $\Lambda[\Gamma]$ can be uniquely written as
          $x + ye_\alpha$ where $x,y$ are products of edges different from
          $\alpha$, that is $x,y$ are both in the image of $i^I_\alpha$.
          The image of $i_\alpha^I$ are the classes for which $y=0$. Since
          \[
            g_\alpha^\Lambda(x+ye_\alpha)=\pi(x\otimes 1+y\otimes e_\alpha)=y\otimes e_\alpha ,
           \]
            the kernel of
          $g^\Lambda_\alpha$ also consists of the classes for which $y=0$. This proves exactness at $\Lambda^r(\Gamma)$. Finally, the map
          $p_\alpha^\Lambda\circ i_\alpha^\Lambda:\Lambda^r(\Gamma\smallsetminus \alpha)\to \Lambda^r(\Gamma/\alpha)$ is an isomorphism, so that for any class $x\otimes e_\alpha\in \Lambda(\Gamma/\alpha)\otimes e_\alpha$ we can find $\bar x\in \Lambda^r(\Gamma\smallsetminus \alpha)$ such that $p_\alpha^\Lambda i_\alpha^\Lambda\bar x=x$ and \[
            g_\alpha^\Lambda(i_ \alpha^\Lambda(\bar x)e_\alpha)=
            \pi(\psi_\alpha^\Lambda(i_\alpha^\Lambda(\bar x)e_\alpha))=
            \pi(p_\alpha^\Lambda i_\alpha^\Lambda \bar x\otimes e_\alpha)=
            \pi(x\otimes e_\alpha)=x\otimes e_\alpha.
          \]
          It follows that $g^\Lambda_\alpha$ is surjective.
	\end{proof}

	\begin{lemma}\label{surjective_g}
		The map $g_\alpha^I$  in diagram~\ref{diagram} is surjective.
	\end{lemma}

	\begin{proof}
              
             The map $g_\alpha^I$ is a map of $\Lambda^r(\Gamma\smallsetminus \alpha)$-modules since 
             \[
              g^I_\alpha((x_1)(x_2+y_2e_\alpha))= 
              g^I_\alpha(x_1x_2 + (x_1y_2)e_\alpha)=x_1y_2\otimes e_\alpha = x_1 g^I_\alpha(x_2 + y_2\alpha)         
              \]
               This means that it is sufficient to prove that each element of a set of generators for $I(\Gamma/\alpha)\otimes e_\alpha$
               as  $\Lambda(\Gamma\smallsetminus \alpha)$-module is in the image of $g_\alpha^I$.
               Note that the map $x \mapsto x\otimes e_\alpha$ defines an isomorphism of $\Lambda(\Gamma\smallsetminus\alpha)$ modules $I(\Gamma/\alpha)\to I(\Gamma/\alpha)\otimes e_\alpha$. It follows that if we define 
               \[
              \bar A(w) = A(w)\otimes e_\alpha
               \]
for circuits $w$ in $\Gamma/ \alpha$, then the classes $\bar A(w)$ form a set of generators for $I(\Gamma/\alpha)\otimes e_\alpha$. We conclude that it suffices to show that for each circuit $w$ in $\Gamma/ \alpha$, the class $\bar A(w)$ is in the image of $g_\alpha^I$.

 Let $v_1,v_2\in V(\Gamma)=V(\Gamma\smallsetminus \alpha)$ be the vertices incident to $\alpha$, and $v \in V(\Gamma/\alpha)$ the vertex given be collapsing $v_1$ and $v_2$.
The vertex $v$ might be incident to some of the edges $[w_i]\in E(\Gamma/\alpha)$. Since we are assuming that $\Gamma$ has no multiple edges or loops, $\Gamma/\alpha$ also has no loops, although it might have double edges. We can decompose the circuit $w$ as a composition of circuits $w^{(i)}$, each starting and ending with the vertex $v$, such that 
\[
w^{(1)}_1,w^{(1)}_2,\cdots w^{(1)}_{l(w^{(1)})},w^{(2)}_1\cdots w^{(2)}_{l(w^{(2)})}\dots w^{(k)}_{l(w^{(k)})}
\]
is a cyclic reordering of $w_1\dots w_{l(w)}$.
A cyclic reordering will at most flip the sign of $A(w)$  so we get that
\[
A(w)=\sum_i \pm e_{w^{(1)}}e_{w^{(2)}}\cdots e_{w^{(i-1)}}A(w^{(i)})e_{w^{(i+1)}}\cdots e_{w^{(k)}}  
\]
This reduces the lemma further to the case when at most two edges of $w$ are incident to $v$.   

Let the circuit be $\{w_1,w_2,\dots w_{l(w)}\}$.
 Since the map
              $p\circ i :E(\Gamma\smallsetminus \alpha) \to E(\Gamma/ \alpha)$ is a bijection, each edge $w_i\in E(\Gamma/\alpha)$ is the image of some
              unique $w_i'\in E(\Gamma\smallsetminus \alpha)$.
If the edges $w'_1,\dots w_{l(w)}'$ form a circuit $w'$ in $\Gamma\smallsetminus \alpha$, then $\bar A(w)=g_\alpha^I(w')$, and we are finished here. 

If the the edges  $w_1',\dots,w_{l(w)}'$ do not form a circuit, this is because
there is an $i$ so that $w_i$ and $w_{i+1}$ are adjacent to $v_1$ and $v_2$ (in either order).
Since $\alpha$ is an edge incident to the vertices $v_1,v_2$, we can form the circuit
 $w''$ to be circuit $w_1',\dots w_i', \alpha,w_{i+1}'\dots w_{l(w)}'$. Then $g_\alpha^I(w'')=\pm \bar A(w)$, and the proof is complete. 
\end{proof}

\begin{corollario}
  \label{le:preliminary_exactness}
  The sequence
  \[
R^r(\Gamma\smallsetminus \alpha)\xrightarrow{i^R_\alpha} R^r(\Gamma) \xrightarrow{g_\alpha^R} R^r(\Gamma/\alpha) \to 0 
\]
is exact.
\begin{proof}
		The map $g_\alpha^R$ is surjective since $g_\alpha^\Lambda$ is surjective.
                That the composite $g_\alpha^Ri_\alpha^R$ is trivial follows from a simple
                diagram chase, using that the composite in
                the middle column is trivial, and that the quotient map
                $\Lambda[\Gamma\smallsetminus \alpha] \to R^r(\Gamma\smallsetminus \alpha)$ is surjective. 
                The only thing left to check is that
                $\mathrm{im}(i_\alpha^R)=\ker(g_\alpha^R)$.

                The columns of the diagram~\ref{diagram} are chain complexes, so that the
                diagram defines a short exact sequence of chain complexes. 
                By lemma~\ref{le:middle}, the homology of the  middle column vanishes.
                Using the long exact sequence of a short exact sequence of chain complexes,
                we see that the quotient 
                $\ker(g_\alpha^R)/\mathrm{im}(i_\alpha^R)$ is isomorphic to the cokernel
                of the map $g_\alpha^I$. According to lemma~\ref{surjective_g}, this cokernel is trivial.
\end{proof}
              \end{corollario}

In preparation for the proof of theorem~\ref{theorem1}, we need a lemma.
\begin{lemma}
  \label{le:commutative}
  Let $\alpha,\beta$ be two different edges of $\Gamma$. The following diagram is commutative
  \begin{equation}
    \label{eq:R}
  \begin{tikzcd}
    R^r(\Gamma\smallsetminus \alpha)\ar[r,"g^R_\beta"]\ar[d,"i_\alpha^R"]&
    R^r(\Gamma\smallsetminus \alpha/\beta)\ar[d,"i_\alpha^R\otimes \mathrm{Id}"]\otimes e_\beta\\
    R^r(\Gamma)\ar[r,"g^R_\beta"]&R^r(\Gamma/\beta)\otimes e_\beta
  \end{tikzcd}
  \end{equation}
\begin{proof}
  In the formulation of the lemma we have tacitely and legitimately identified the graph $(\Gamma\smallsetminus \alpha)/\beta$ with the graph
  $(\Gamma/\beta)\smallsetminus \alpha$.
  We first note the commutativity of the diagram
  \begin{equation}
    \label{eq:Lambda1}
    \begin{tikzcd}
    \Lambda^r(\Gamma\smallsetminus \alpha)\ar[r,"\psi^\Lambda_\beta"]\ar[d,"i_\alpha^R"]&
    \Lambda^r(\Gamma\smallsetminus \alpha/\beta)\ar[d,"i_\alpha^\Lambda\otimes \mathrm{Id}"]\otimes \Lambda^r[e_\beta]   \\
    \Lambda^r(\Gamma)\ar[r,"\psi^\Lambda_\beta"]&\Lambda^r(\Gamma/\beta) \otimes \Lambda^r[e_\beta]   
    \end{tikzcd}
   \end{equation}
   Since $\psi_\beta^\Lambda$ and $i_\alpha^\Lambda$ are ring maps, it suffices to check this on generators $e_\eta$, which is trivial to do.
   Applying the projection $\pi$, we obtain that the following diagram is commutative:
 \begin{equation}
    \label{eq:Lambda2}
    \begin{tikzcd}
    \Lambda^r(\Gamma\smallsetminus \alpha)\ar[r,"g^\Lambda_\beta"]\ar[d,"i_\alpha^R"]&
    \Lambda^r(\Gamma\smallsetminus \alpha/\beta)\ar[d,"i_\alpha^\Lambda\otimes \mathrm{Id}"]\otimes e_\beta   \\
    \Lambda^r(\Gamma)\ar[r,"g^\Lambda_\beta"]&\Lambda^r(\Gamma/\beta) \otimes e_\beta   
    \end{tikzcd}
   \end{equation}

The statement of the lemma follows from that there is a surjective map from diagram~\ref{eq:Lambda2} to diagram~\ref{eq:R},   
\end{proof}
\end{lemma}

Note that if $\Gamma$ does not have any loops,
the ring map $c:\Lambda[\Gamma]\to \ZZ$, $c(1)=1$ and $ c(e_\alpha)=0$ for all edges $\alpha$ in $\Gamma$ factors over $R^r(\Gamma)$. This is not the case if $\Gamma$ has a loop $\alpha$ because the Arnold relation corresponding to the circuit consisting of he single edge $\alpha$ is not mapped to 0 by $c$.   
It follows that if $\Gamma$ has no loops, the canonical map $\ZZ\to R^r(\Gamma)$ is a split inclusion, with left inverse the map $\eta$ that maps each $e_\alpha$ to 0.
We say that a graph $\Gamma$ satisfies  $(*)$ if both of the following two statement are true.
\begin{itemize}
\item If $\Gamma$ has no loops, for each $\alpha\in E(\Gamma)$ the map
  $i_\alpha:R^r(\Gamma\smallsetminus \alpha)\to R^r(\Gamma)$ is injective.
\item If $\Gamma$ does not have loops or multiple edges and if $x\in R^r(\Gamma)$ and $x\neq \ZZ$, there exists a $\beta\in E(\Gamma)$ such that $g^R_\beta(x)\neq 0$.  
\end{itemize}

\begin{lemma}
  \label{le:induction_step}
Every graph $\Gamma$ satisfies $(*)$.  
  \begin{proof}
We will argue by induction on the number of edges of $\Gamma$.
The graph with one vertex and no edges satisfies $(*)$ for trivial reasons.
 
The induction hypothesis is that every graph with at most $n-1$ edges satisfies $(*)$. Let $\Gamma$ be a graph with $n$ edges. We need to show that $\Gamma$ satisfies $(*)$.

We first show that $i_\alpha^R:R^r(\Gamma\smallsetminus \alpha)\to R^r(\Gamma)$ is injective.
Using lemma~\ref{le:double_edges} we easily reduce to the case that $\Gamma$ has no multiple edges.
The map $i_\alpha$ preserves the direct sum decomposition $R^r(\Gamma)\cong \ZZ  \oplus \ker(\eta)$, so it
suffices to show that if   
$\alpha\in E(\Gamma)$, $x\in R^r(\Gamma)\setminus \ZZ$, then $i_\alpha^R(x)\neq 0$.

Since $\Gamma\smallsetminus \alpha$ has no multiple edges and satisfies $(*)$ by assumption, there is an edge $\beta\in E(\Gamma\smallsetminus \alpha)$ such that $g_\beta^R(x)\neq 0\in R^r(\Gamma\smallsetminus\alpha/\beta)$. Since $\Gamma/\beta$ has no loops, it satisfies $(*)$, $i_\alpha^Rg_\beta^R(x)\neq 0$.
Now apply lemma~\ref{le:commutative} to prove that $i^R_\alpha(x)\neq 0\in R^r(\Gamma)$ as required.

We finally need to prove that if $\Gamma$ has no multiple edges,
$x\in R^r(\Gamma)\setminus \ZZ$ and
$g^R_\eta(x)=0$ for all $\eta\in E(\Gamma)$, then $x=0$. Pick any $\alpha\in E(\Gamma)$. Since
$g^R_\alpha(x)=0$, by lemma~\ref{le:preliminary_exactness} there is an
$y\in R^r(\Gamma\smallsetminus \alpha)$ such that $i_\alpha^R(y)=x$.
Using lemma~\ref{le:commutative} again, we see that
for any $\beta\in E(\gamma\smallsetminus \alpha)$:
\[
(i^R_\alpha\otimes \mathrm{Id}) g_\beta^R(y)=g^R_\beta i_\alpha^R(y)=g_\beta^R(x)=0
\]
Because $\Gamma/\beta$ satisfies $(*)$, and because $(\Gamma\smallsetminus\alpha)/\beta$
either equals $\Gamma/\beta$ or $(\Gamma/\beta)\smallsetminus \alpha$
the map
$i^R_\alpha\otimes \mathrm{Id}:R^r(\Gamma\smallsetminus\alpha/\beta)\to R^r(\gamma/\beta)$ is injective, so that $g_\beta^R(y)=0$. Because this is true for every
$\beta\in E(\Gamma\smallsetminus \alpha)$, and $\Gamma\smallsetminus \alpha$ 
 also satisfies $(*)$, it follows that $y=0$ so that $x=0$.
\end{proof}
\end{lemma}

We sum up in
\begin{teorema}
  The columns of diagram~\ref{diagram} are short exact.
  \begin{proof}
    The middle column is exact by lemma~\ref{le:middle}. The
    right hand column is exact by lemma~\ref{le:preliminary_exactness} and  lemma~\ref{le:induction_step}. The exactness of the left hand column follows from this by the nine-lemma (or by simple diagram chase).
  \end{proof}
\end{teorema}

\begin{proof}[Proof of theorem~\ref{theorem1}]
  The injectivity follows from lemma~\ref{le:induction_step}.
  The map $\psi_\alpha^R$ is surjective and the map $\iota_\alpha^R$ is injective, so it suffices to show that if $\psi_\alpha^R(x)\in \mathrm{im}{\iota_\alpha^R}$, then $x \in \mathrm{im}(i_\alpha^R)$. But
  $\psi_\alpha^R(x)\in \mathrm{im}{\iota_\alpha^R}$ if and only if
  $g_\alpha^R(x)=0$, so the theorem follows from the exactness of the right column in diagram~\ref{diagram}.
\end{proof}

	\section{Deletion-contraction in the space
          $\Conf_{r}(\Gamma)$}
        \label{sec:Space}
   In this section we will prove that  there is an isomorphism between the ring $R^r(\Gamma)$ defined in the previous section and the cohomology ring of $\Conf_r(\Gamma)$. Moreover, \textit{Lemma }\ref{lemma surjective} provides the existence of a short exact
   	sequence of the form
   	\begin{equation*}
   	0\rightarrow
   	H^{\ast}(\Conf_r(\Gamma\smallsetminus e))\rightarrow
   	H^{\ast}(\Conf_r(\Gamma)) \rightarrow
   	H^{\ast-r+1}(\Conf_{r}(\Gamma/e))\rightarrow 0.
         \end{equation*}
   	The first step is to describe the deletion-contraction long   exact sequence that occurs for configuration spaces.
\subsection{The long exact sequence}


	We will prove the following theorem.
	\begin{teorema}\label{del-contr space}
		There is a long exact sequence in cohomology
		\begin{salign*}
			\cdots&\longrightarrow H^{\ast}(\Conf_r(\Gamma\smallsetminus e))\longrightarrow
			H^{\ast}(\Conf_r(\Gamma))\\
			&\longrightarrow H^{\ast-r+1}(\Conf_r(\Gamma/e))\longrightarrow H^{\ast+1}(\Conf_r(\Gamma\smallsetminus e))\longrightarrow\cdots
		\end{salign*}
	\end{teorema}
        Before we turn to the proof, we make a few preliminary observations.
         
		Let $e$ be an edge in $\Gamma$ between the vertices $a$ and
		$b$. The space $\Conf_{r}(\Gamma)$ is an open subspace of
		$\Conf_{r}(\Gamma\smallsetminus e)$. The complement
		\begin{equation*}
		A_{e}(\Gamma)=\Conf_{r}(\Gamma\smallsetminus e)-\Conf_{r}(\Gamma)
		\end{equation*}
		is a closed subspace in $\Conf_{r}(\Gamma\smallsetminus e)$ and
		\begin{equation*}
		A_{e}(\Gamma)=\{(x_{1},\dots,x_{n})\in\mathbb{R}^{rn};x_{i}\neq x_{j} \text{ if } \alpha_{i,j}\in E(\Gamma)\smallsetminus\{e\} \text{ while } x_{a}= x_{b}\}.
		\end{equation*}
  		There is a canonical homeomorphism between $A_{e}(\Gamma)$ and
		$\Conf_{r}(\Gamma/e)$ sending $(x_{1},\dots,x_{n})$ to
		$(x_{1},\dots, x_{a},\dots,\widehat{x_{b}},\dots,x_{n})$.
                Let
                $m_{a,b}(x)>0$
                be the minimum of the numbers $|x_a-x_c|$ such that $c\not = b$, but $c$ is connected by an edge to $a$. This number will be independent of $x_b$. 
		We define an open neighborhood $V_{e}(\Gamma)$ of
		$A_e(\Gamma)$ in $\Conf_{r}(\Gamma\smallsetminus e)$ in the following way
		\begin{equation*}
		V_{e}(\Gamma)=\{x=(x_{1},\dots,x_{n})\in \Conf_{r}(\Gamma\smallsetminus e);|x_{a}-x_{b}|<\tfrac 12 m_{a,b}(x) \}
		\end{equation*} 
                	\begin{lemma}\label{lemma hom equi}
		$\Conf_r(\Gamma)\cap V_{e}$ is homotopy equivalent to
		$\mathbb{S}^{r-1}\times \Conf_{r}(\Gamma/e)$.
	\end{lemma}
	\begin{proof}
		$\Conf_r(\Gamma)\cap V_{e}$ is the space
		\begin{equation*}
		\{(x_{1},\dots,x_{n})\in \Conf_{r}(\Gamma):0<|x_{a}-x_{b}|<\tfrac 12 m_{a,b}(x)\}
		\end{equation*}
		We define the maps
		\begin{equation*}
		f:\Conf_r(\Gamma)\cap V_{e}\rightarrow\mathbb{S}^{r-1}\times \Conf_{r}(\Gamma/e)
		\end{equation*}
		by
		\begin{equation*}
		f((x_1,\dots,x_n))=\left(\frac{x_a-x_b}{|x_a-x_b|},(x_1,\dots,x_a,\dots,\widehat{x_b},\dots,x_n)\right)
		\end{equation*}
		and
		\begin{equation*}
		g:\mathbb{S}^{r-1}\times \Conf_{r}(\Gamma/e)\rightarrow \Conf_r(\Gamma)\cap V_{e}
		\end{equation*}
		by
		\begin{equation*}
		g(y, (x_{1},\dots,x_{n}))=\left(x_{1},\dots,x_a,\dots,x_a+m_{a,b}(x) y,\dots,x_{n}\right)
		\end{equation*}
		Now $gf$ is clearly homotopic to the identity and $fg$
		equals the identity.
	\end{proof}

        \begin{proof}[Proof of theorem~\ref{del-contr space}]
		We have two open subspaces $\Conf_{r}(\Gamma)$ and
		$V_{e}(\Gamma)$ of $\Conf_{r}(\Gamma \smallsetminus e)$ such that
		$\Conf_{r}(\Gamma)\cup
		V_{e}(\Gamma)=\Conf_{r}(\Gamma \smallsetminus e)$. There is a pushout
		diagram
		
		\begin{equation*}
		\begin{tikzcd}
		&V_e(\Gamma)\cap \Conf_r(\Gamma)\ar[r]\ar[d]& V_e(\Gamma)\ar[d]&\\
		&\Conf_r(\Gamma)\ar[r]&\Conf_r(\Gamma\smallsetminus e)&\\
		\end{tikzcd}
		\end{equation*}
		We obtain a Mayer-Vietoris long exact sequence in
		cohomology
		\begin{equation*}
		\begin{split}
		\cdots&\xrightarrow{\ \hphantom{\psi^\ast}\ }
		H^{\ast}(\Conf_r(\Gamma\smallsetminus e))\xrightarrow{\ \phi^\ast\ }
		H^{\ast}(\Conf_r(\Gamma))\oplus
		H^{\ast}(V_{e}(\Gamma))
		\\
		&\xrightarrow{\ \psi^\ast\ } H^{\ast}(\Conf_r(\Gamma)\cap
		V_{e}(\Gamma))\xrightarrow{\ \delta^{\ast}\ }
		H^{\ast+1}(\Conf_r(\Gamma\smallsetminus e))\xrightarrow{\ \hphantom{\psi^\ast}\ }\cdots
		\end{split}
		\end{equation*}
		where $\phi$ is the map assigning to each cohomology class
		$x$ its restrictions
		$(x|_{\Conf_{r}(\Gamma)},x|_{V_{e}(\Gamma)})$ and
		$\psi(x,y)=x-y$.
		
		We notice that $V_{e}(\Gamma)$ is homotopy equivalent to
		$\Conf_{r}(\Gamma/e)$ and by \textit{Lemma} \ref{lemma hom equi}
		$\Conf_r(\Gamma)\cap V_{e}$ is homotopy equivalent to
		$\mathbb{S}^{r-1}\times \Conf_{r}(\Gamma/e)$. Let $[\mu]$
		denote the fundamental class of $\mathbb{S}^{r-1}$. By the
		Kunneth formula, we can rewrite the long exact sequence as
		\begin{equation*}
		\begin{split}
		\cdots&\longrightarrow
		H^{\ast}(\Conf_r(\Gamma\smallsetminus e))\longrightarrow
		H^{\ast}(\Conf_r(\Gamma))\oplus
		H^{\ast}(\Conf_{r}(\Gamma/e))
		\\
		&\longrightarrow\bigoplus_{k+l=\ast}
		H^{k}(\mathbb{S}^{r-1})\otimes
		H^{l}(\Conf_{r}(\Gamma/e))\longrightarrow
		H^{\ast+1}(\Conf_r(\Gamma\smallsetminus e)\longrightarrow\cdots
		\end{split}
		\end{equation*}
		This implies the existence of the long exact sequence
		\begin{equation*}
		\begin{split}
		\cdots&\longrightarrow H^{\ast}(\Conf_r(\Gamma \smallsetminus e))\longrightarrow
		H^{\ast}(\Conf_r(\Gamma))
		\\
		&\longrightarrow [\mu]
		H^{\ast}(\Conf_{r}(\Gamma/e))\longrightarrow
		H^{\ast+1}(\Conf_r(\Gamma \smallsetminus e))\longrightarrow\cdots
		\end{split}
		\end{equation*}
		Finally, using the isomorphism
		$[\mu] H^{\ast-r+1}(\Conf_{r}(\Gamma/e))\cong
		H^{\ast-r+1}(\Conf_{r}(\Gamma/e))$ we have the
		deletion-contraction long exact sequence for generalised
		configuration spaces:
		\begin{salign*}
			\cdots&\longrightarrow H^{\ast}(\Conf_r(\Gamma\smallsetminus e))\longrightarrow
			H^{\ast}(\Conf_r(\Gamma))
			\\
			&\longrightarrow
			H^{\ast-r+1}(\Conf_{r}(\Gamma/e))\longrightarrow
			H^{\ast+1}(\Conf_r(\Gamma\smallsetminus e))\longrightarrow\cdots
		\end{salign*}
              \end{proof}

\subsection{The map from $R^r(\Gamma)$.}
Let $\Gamma$ be a graph and $r$ a natural number. 
 For any edge
	$e=e(v_1,v_2)\in E(\Gamma)$, ordered by that $v_1<v_2$, there is a map
	\begin{equation*}
	p_e:\Conf_r(\Gamma)\to S^{r-1}
	\end{equation*}
	defined by
	\begin{equation*}
	p_{e}(x)\mapsto\frac{x_{v_2}-x_{v_1}}{\vert{x_{v_2}-x_{v_1}}\vert}\in S^{r-1}\subset \mathbb{R}^{r}\setminus \{0\}.
	\end{equation*} 
If all edges in $\Gamma$ have an orientation, we can combine these maps to a map
	\begin{equation*}
	p(\Gamma):\Conf_r(\Gamma)\to (S^{r-1})^{E(\Gamma)}.
	\end{equation*}
	We choose a standard generator
	$[S^{r-1}]\in H^{r-1}(S^{r-1})$. After choosing a total order of the edges,  we can identify $H^*((S^{r-1})^{E(\Gamma)})$ with the ring $\Lambda [E(\Gamma)]$. If $r$ is even, this identification depends of the order of the edges, but not on the orientation of the edges. If $r$ is odd, the identification depends on the orientation of the edges, but not of the order of the edges. In both cases, two different choices differ by an isomorphism. 

	\begin{definizione}
		Let $r$ be an even number, the maps $p_r(\Gamma)$ induce
		ring homeomorphisms
		\begin{equation*}
		p_r(\Gamma)^*:\Lambda[\Gamma]\to H^*(\Conf_r(\Gamma));
		\quad p_r(\Gamma)(e)=p_{e}^*([S^{r-1}]).
		\end{equation*}
	\end{definizione}
	\begin{lemma}\label{lemma surjective}
          The map $p_r(\Gamma)^*$ is surjective.
         		There is a short exact sequence
		\begin{equation*}
		0\rightarrow H^{\ast}(\Conf(_r\Gamma\smallsetminus e))\rightarrow
		H^{\ast}(\Conf_r(\Gamma)) \rightarrow
		H^{\ast-r+1}(\Conf_{r}(\Gamma/e))\rightarrow 0.
		\end{equation*} 
	\end{lemma}
	\begin{proof}
		We prove it by induction on the number of edges in
		$\Gamma$. The lemma is true if $\Gamma$ has one edge. Now we
		suppose the result true for graphs with $n-1$ edges. Chose one edge $\alpha\in E(\Gamma)$.

                If we have made choices of orientation of edges and order of edges for $\Gamma$, we can make compatible choices for $\Gamma\smallsetminus \alpha$ respectively $\Gamma/\alpha$, so that the maps $i:E(\Gamma\smallsetminus\alpha)\to E(\Gamma)$ respectively
        $p:E(\Gamma)\to E(\Gamma/\alpha)$ preserve the orientations and orders of the edges.  Assume that we have made such compatible choices.

We have a commutative diagram
		\begin{equation*}
		\xymatrix{
			0 \ar[r] & \Lambda[\Gamma\smallsetminus \alpha] \ar[d]_-{p_r(\Gamma\smallsetminus \alpha)^*} \ar[r]^-{i^\Lambda_\alpha} & \Lambda[\Gamma]\ar[d]_{p_r(\Gamma)^*}  \ar[r]^-{p^\Lambda_\alpha} & \Lambda[\Gamma/\alpha]\ar[d]^-{p_r(\Gamma/\alpha)^*} \ar[r] & 0 \\
			\cdots \ar[r] & H^{\ast}(\Conf_r(\Gamma\smallsetminus \alpha))
			\ar[r]_-{\phi^\ast} & H^{\ast}(\Conf_r(\Gamma))
			\ar[r]_-{\psi^\ast} & H^{\ast-r+1}(\Conf_r(\Gamma/\alpha)) \ar[r]
			& \cdots }
		\end{equation*}
		The first and last vertical maps are surjective by
		the induction hypothesis and $p^\Lambda_\alpha$ is also surjective by  lemma~\ref{le:middle}. By the commutativity
		of the diagram
		$p_r(\Gamma/e)\circ p^\Lambda_\alpha=\psi^\ast\circ p_r(\Gamma)$. Moreover
		$p_r(\Gamma/e)\circ p^\Lambda_\alpha$ is surjective since it is the
		composition of surjective maps. It follows that $\psi^\ast$
		is surjective, so that the long exact sequence at the bottom row breaks up into short exact sequences. Therefore the diagram above is a map of short exact sequences, and it follows by the five lemma that the middle vertical map is surjective.
	\end{proof}
	\begin{lemma}\label{map to 0}
		The map $p_r(\Gamma)^*$ maps elements in the ideal generated
		by the generalised Arnold relations to $0$.
	\end{lemma}
	\begin{proof}
          It suffices to show that if $w$ is a circuit in $\Gamma$, and if $A(w)\in \Lambda[\Gamma]$ is the corresponding Arnold element, then $p_r(\Gamma)^*(A(w))=0$. Let $C_n$ be a cyclic graph with vertices $v_1,v_2,\dots v_n$ and edges
          $c_i=e(v_i,v_{i+1})$ for $i\leq n-1$ together with $c_n=e(v_{n},v_1) $. The edges of $C_n$ form a circuit $c^n$. There is a map of graphs $f:C_{l(w)}\to\Gamma$ which maps the edge $c_i^n\in E(C_n)$ to $w_i\in E(\Gamma)$.
          This map induces $f_*^\Lambda:\Lambda[C_n]\to \Lambda[\Gamma]$ and
          $(f^{conf})^*:  H^*(\Conf_{r}(C_r))\to  H^*(\Conf_{r}(\Gamma))$.
          By naturality, there is a commutative diagram:
\[
    \begin{tikzcd}
      \Lambda^r(C_n) \ar[r,"f_*\Lambda"]\ar[d,"p_r(C_n)^*"] &  \Lambda^r(\Gamma) \ar[d,"p_r(\Gamma)^*"]\\
      H^*(\Conf_{r}(C_r))\ar[r,"(f^{conf})^*"]&  H^*(\Conf_{r}(\Gamma))
    \end{tikzcd}
\] 
Using that $A(w)=f_*^\Lambda (A(c^n))$, it follows from this diagram
that it suffices to show that for every $n$, the Arnold class $A(c^n)$ is in the kernel of the map $p_r(C_{n})^*$.

In order to prove the lemma, we investigate the kernel of the
map of classes in degree $(n-1)(r-1)$.  We will prove inductively that the kernel of $p_r(C_n)^*:(\Lambda[C_r])^{(n-1)(r-1)}\to H^{(n-1)(r-1)}(\Conf_{r}(C_r))$ is generated by  the Arnold element $A(c^n)$.         
Note that the group $(\Lambda(C_n))^{(n-1)(r-1)}\cong \ZZ^n$ is generated by the classes
 $c^n_1c^n_2\cdots \widehat{c^n_i}\cdots c^n_n$.

We make a preliminary remark. Let $I_n$ be the linear graph
$C_n\smallsetminus c_n$. Then 
		\begin{equation*}  
		\Conf_{r}(I_n)=\{(x_{1},\dots,x_{n})\in\mathbb{R}^{kr}: x_{1}\neq x_{2},\dots,x_{n-1}\neq x_{n}\}
              \end{equation*}
              The map  $p(I_n)$ is a homotopy equivalence, and $p_r(I_n)^*$ an isomorphism. In particular $H^{(n-1)(r-1)}(\Conf_{r}(I_n))\cong \ZZ$, generated by the class $p_r(I_r)^*(c^n_1c^n_2\cdot c^n_{n-1})$.
              
                Let $n=2$. In this case $C_2$ has a double edge, so that
                $\Conf_{r}(I_2)\to \Conf_{r}(C_2)$ is a homeomorphism,
                $A(c)_2 =\pm c_1\pm c_2$ and 
                \[
                  p_r(C_2)^*:(\Lambda[C_2]/\{A(c^2)=0\})^{(r-1)}
                  \cong H^{(r-1)}(\Conf_{r}(C_2)),
                \]
                so that we have an induction start.

Let $n\geq 3$, and assume the induction hypothesis for $n-1$. Let $c_i$ be any edge of $C_n$, so that $C_n\smallsetminus c_i$ is isomorphic to $I_n$. Let $m=(n-1)(r-1)$. There is a surjective map of short exact sequences
\[
  \begin{tikzcd}
    \Lambda[C_n\smallsetminus c_i]^{(m)}\ar[r,"i^\Lambda_{c_i}"]\ar[d,"p_r(C_n\smallsetminus c_i)^*"] & \Lambda[C_n]^{(m)}\ar[r,"p^\Lambda_{c_i}"]\ar[d,"p_r(C_n)^*"] & \Lambda[C_n/c_i]^{(m)}\ar[d,"p_r(C_n/c_i)^*"] \\
    H^{(m)}(\Conf_{r}(C_n\smallsetminus c_i))\ar[r,"\phi^*"]&H^{(m)}(\Conf_{r}(C_n))\ar[r,"\psi^*"]&H^{(m)}(\Conf_{r}(C_n/c_i))
  \end{tikzcd}
\]  

Because the rows are short exact, we have an induced exact sequence of kernels and cokernels:
\[
  \ker (p_r(C_n\smallsetminus c_i)^*)
  \xrightarrow{i^\Lambda_{c_i}} \ker( p_r(C_n)^*)
  \xrightarrow{p^\Lambda_{c_i}} \ker (p_r(C_n/ c_i)^*)
  \xrightarrow{\partial} \mathrm{coker} (p_r(C_n\smallsetminus c_i)^*)
\]
The map $p_r(C_n\smallsetminus c_i)^*$ is an isomorphism by the preliminary remark. Therefore $p^\Lambda_{c_i}$ restricts to an isomorphism $\ker(p_r(C_n)^*)\to \ker(p_r(C_n/c_i^n)^*)\cong \ZZ$.
Notice also that $p^\Lambda_{c_i}(A(c^n))=A(c^{n-1})$, where we in the notation have identified $C_n/c_i$ with $C_{n-1}$. In order to complete the proof, we only need to show that $A(c^n)\in \ker (p_r(C_n)^*)$. 

By the diagram and the inductive assumption, $p_r(C_n)^*(A(c^n))$ is in the image of the map $\phi^*$, so there is an $x\in \Lambda[C_n\smallsetminus c_i]$ such that
$p_r(C_n)^*(A(c^n)-i^\Lambda_{c_i}(x))\in \ker p_r(C_n)^*$.
We conlude: For every $i$, $1\leq i\leq n$ there is a number $n_i$ such that
\begin{equation}
  \label{eq:ker1}
p_r(C_n)^*(A(c^n)-i^\Lambda_{c_i}(n_ic^n_1\cdots \widehat{c^n_i}\cdots c^n_n))=0
\end{equation}
We need to show that $n_i=0$. To prove this, we pick $j\neq i$.
\begin{align*}
  p_r(C_n/c_i^n)^*p_{c^n_i}^\Lambda(c^n)&(A(c^n)-i^\Lambda_{c_i}(n_ic_1\cdots \widehat{c_i}\cdots c_n))\\
                                        &= p_r(C_n/c_i^n)^*((A(c^{n-1})-n_ic_1 c_1\cdots \widehat{c_i}\cdots \widehat{c_j}\dots c_n\\
  &=-p_r(C_n/c_1)^*(n_ic_1^n\cdots \widehat{c_i}\cdots \widehat{c_j}\dots c_n).
\end{align*}
Since by the inductive assumption the kernel of $p_r(C_n/c_i)^*$ is the subgroup generated by $A(c^{n-1})$, and since $n_ic_1^n\cdots \widehat{c_i}\cdots \widehat{c_j}\dots c_n$ is only in this subgroup if $n_i=0$, we have proved   
\[
p_{c_j}^\Lambda(i^\Lambda_{c_i}(n_ic_i\cdots \widehat{c_i}\cdots c_n))=c_1c_2\cdots \widehat{c_i}\cdots \widehat{c_j}\cdots c_n
\]
is not in the subgroup of $\Lambda[C_n/c_j]$ generated by $A(c_{n-1})$, since $n-1>1$. It follows from the induction hypothesis that $n_i=0$, so that
$p_r(C_n)^*(A(c^n))=0$. This finishes the proof.
	\end{proof}
	\begin{corollario}
		The map
		$p_r(\Gamma)^*:\Lambda[\Gamma]\to H^*(\Conf_r(\Gamma))$
		factors uniquely over the map
		$\lambda_{r}:R^{r}(\Gamma)\rightarrow
		H^{\ast}(\Conf_{r}(\Gamma))$.
	\end{corollario}
	
	\begin{teorema}
		There is a isomorphism of graded commutative rings
		\begin{equation*}
		\lambda_{r}:R^{r}(\Gamma)\rightarrow H^{\ast}(\Conf_{r}(\Gamma)).
		\end{equation*}
	\end{teorema}
	\begin{proof}
		We prove the theorem by induction on the number of
		edges in the graph. Assume that the lemma is true for
                all graphs with $n-1$ or fewer edges.  
                Let $\Gamma$ be graph with $n$  edges. If $\Gamma$ has multiple
                edges, the lemma follows from  the induction hypothesis and
                lemma~\ref{le:double_edges}. Consider the following map of short exact sequences:
		\begin{equation*}
		\xymatrix{
			0 \ar[r] & R^{r}(\Gamma\smallsetminus e) \ar[d]_-{\cong} \ar[r]^-{} & R^{r}(\Gamma)\ar[d]  \ar[r]^-{} & R^{r}(\Gamma/e) \ar[d]^-{\cong} \ar[r] & 0 \\
			0 \ar[r] & H^{\ast}(\Conf_r(\Gamma\smallsetminus e))
			\ar[r]_-{\phi^\ast} & H^{\ast}(\Conf_r(\Gamma))
			\ar[r]_-{\psi^\ast} & H^{\ast}(\Conf_r(\Gamma/e)) \ar[r]
			& 0\underline{} }
		\end{equation*}
		By the five lemma if follows that the middle map is also an
		isomorphism.
	\end{proof}

\bibliographystyle{plainnat}
\bibliography{progress}	
	
\end{document}